\documentclass{article}
\usepackage[utf8]{inputenc}
\usepackage{amsthm,amssymb,geometry}

\usepackage{tikz}
\usepackage[colorinlistoftodos]{todonotes}

\title{Restricted Stacks as Functions}
\author{Katalin Berlow}
\date{\today}

\newcommand{\sort}{\operatorname{sort}}
\newcommand{\id}{\operatorname{id}}

\newcommand{\N}{\mathbb{N}}
\newcommand{\av}[1]{\textrm{Av}_n(#1)}
\usepackage{amsmath,amsthm,amssymb,amsbsy,graphicx,comment,thmtools}

\usepackage{hyperref}
\usepackage[capitalize]{cleveref}

\newcommand{\Z}{\mathbb{Z}}

\newtheorem{theorem}{Theorem}[section]
\newtheorem{lemma}[theorem]{Lemma}
\newtheorem{cor}[theorem]{Corollary}
\newtheorem{conj}[theorem]{Conjecture}
\newtheorem{ques}[theorem]{Open Question}
\theoremstyle{definition}
\newtheorem{definition}[theorem]{Definition}

\begin{document}

\maketitle
\begin{abstract}
The stack sort algorithm has been the subject of extensive study over the years. In this paper we explore a generalized version of this algorithm where instead of avoiding a single decrease, the stack avoids a set $T$ of permutations. We let $s_T$ denote this map. We classify for which sets $T$ the map $s_T$ is bijective. A corollary to this answers a question of Baril, Cerbai, Khalil, and Vajnovszki about stack sort composed with $s_{\{\sigma,\tau\}}$, known as the $(\sigma,\tau)$-machine. This fully classifies for which $\sigma$ and $\tau$ the preimage of the identity under the $(\sigma,\tau)$-machine is counted by the Catalan numbers.  We also prove that the number of preimages of a permutation under the map $s_T$ is bounded by the Catalan numbers, with a shift of indices.  For $T$ of size 1, we classify exactly when this bound is sharp. We also explore the periodic points and maximum number of preimages of various $s_T$ for $T$ containing two length $3$ permutations. 
\end{abstract}

\section{Introduction}

 Consider the problem of sorting permutations using the following algorithm which maintains a stack. If the stack is currently empty or if the leftmost element of the input is smaller than the top element of the stack, then move the leftmost element onto the stack. Otherwise, move the top element off the stack and append it to the output. This algorithm is called \textit{stack sort}. We will denote the map by $s$. It can be visualised as follows: 

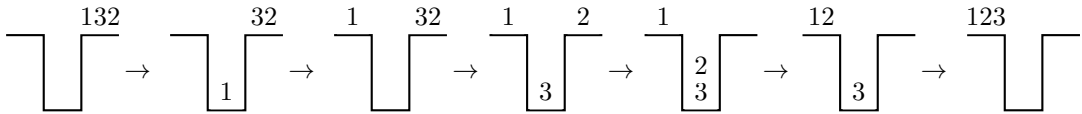
\begin{figure}[h]\begin{center}
\begin{tikzpicture}[scale=0.5]
				\draw[thick] (0,0) -- (1,0) -- (1,-2) -- (2,-2) -- (2,0) -- (3,0);
				\node[fill = white, draw = white] at (2.5,.5) {132};
				\node[fill = white, draw = white] at (1.5,-1.5) {};
				\node[fill = white, draw = white] at (1.5,-.8) {};
				\node[fill = white, draw = white] at (3.5,-1) {$\rightarrow$};
              \end{tikzpicture}
         \begin{tikzpicture}[scale=0.5]
				\draw[thick] (0,0) -- (1,0) -- (1,-2) -- (2,-2) -- (2,0) -- (3,0);
				\node[fill = white, draw = white] at (2.5,.5) {32};
				\node[fill = white, draw = white] at (1.5,-1.5) {1};
				\node[fill = white, draw = white] at (1.5,-.8) {};
					\node[fill = white, draw = white] at (3.5,-1) {$\rightarrow$};
              \end{tikzpicture} \begin{tikzpicture}[scale=0.5]
				\draw[thick] (0,0) -- (1,0) -- (1,-2) -- (2,-2) -- (2,0) -- (3,0);
				\node[fill = white, draw = white] at (2.5,.5) {32};
				\node[fill = white, draw = white] at (1.5,-1.5) {};
				\node[fill = white, draw = white] at (.5,.5) {1};
             	\node[fill = white, draw = white] at (3.5,-1) {$\rightarrow$}; \end{tikzpicture}\begin{tikzpicture}[scale=0.5]
				\draw[thick] (0,0) -- (1,0) -- (1,-2) -- (2,-2) -- (2,0) -- (3,0);
				\node[fill = white, draw = white] at (2.5,.5) {2};
				\node[fill = white, draw = white] at (1.5,-1.5) {3};
				\node[fill = white, draw = white] at (.5,.5) {1};
             	\node[fill = white, draw = white] at (3.5,-1) {$\rightarrow$}; \end{tikzpicture}\begin{tikzpicture}[scale=0.5]
				\draw[thick] (0,0) -- (1,0) -- (1,-2) -- (2,-2) -- (2,0) -- (3,0);
				\node[fill = white, draw = white] at (2.5,.5) {};
				\node[fill = white, draw = white] at (1.5,-1.5) {3};
				\node[fill = white, draw = white] at (1.5,-.8) {2};
				\node[fill = white, draw = white] at (.5,.5) {1};
             	\node[fill = white, draw = white] at (3.5,-1) {$\rightarrow$}; \end{tikzpicture}\begin{tikzpicture}[scale=0.5]
				\draw[thick] (0,0) -- (1,0) -- (1,-2) -- (2,-2) -- (2,0) -- (3,0);
				\node[fill = white, draw = white] at (2.5,.5) {};
				\node[fill = white, draw = white] at (1.5,-1.5) {3};
				\node[fill = white, draw = white] at (.5,.5) {12};
             	\node[fill = white, draw = white] at (3.5,-1) {$\rightarrow$}; \end{tikzpicture}\begin{tikzpicture}[scale=0.5]
				\draw[thick] (0,0) -- (1,0) -- (1,-2) -- (2,-2) -- (2,0) -- (3,0);
				\node[fill = white, draw = white] at (.5,.5) {123};
              \end{tikzpicture}
\end{center}
\caption{The stack sorting algorithm at various points in the process of sorting the permutation 132. The ``well" drawn above is known as the \textit{stack}, which must stay increasing while read from top to bottom. This is an example of a permutation which is properly sorted by the stack sort algorithm. }
\end{figure}

\begin{definition}
Let $\sigma$ and $\tau$ be length $n$ words over the alphabet $[n] = \{1,2,\dots, n\}$. We say that $\sigma$ is order isomorphic to $\tau$, in symbols $\sigma \cong \tau$, if for all $i,j\leq n$ we have $$\sigma(i)<\sigma(j) \textrm{ if and only if } \tau(i)<\tau(j) $$  
and  $$\sigma(i)>\sigma(j) \textrm{ if and only if } \tau(i)>\tau(j).$$
\end{definition}

Recall that a permutation of length $n$ is a word of length $n$ with no repeated letters over the alphabet $[n]$. 

\begin{definition}
Let $\sigma$ be a word of length $n$ over the alphabet $[n] = \{1,2,\dots, n\}$ and $\tau$ a permutation of length $m$. We say that $\sigma$ \textit{contains} $\tau$ when $\sigma$ has a (not necessarily contiguous) subsequence which is order-isomorphic to $\tau$. 
Otherwise we say that $\sigma$ is $\tau$\textit{-avoiding}. We say that $\sigma$ is $T$\textit{-avoiding}, for a set $T$ of permutations, when $\sigma$ avoids every element of $T$. 

\end{definition}

For example, the permutations $132456$, $243561$ and $125643$ all contain $132$, while $45312$ does not.  It is well known that the number of permutations of length $n$ which avoid $\sigma\in S_3$, denoted $\av{\sigma}$, is counted by the $n$th Catalan number, $C_n$.

The stack sort map $s$, inspired by Knuth's stack sorting machine \cite{Knuth}, was first formalized by West in \cite{West}. This map has been widely studied, as have its generalizations. In 1968, Knuth proved (see \cite{Knuth}) that the set of permutations of length $n$ correctly sorted to the identity is $\av{231}$, thus showing that the number of permutations sorted to the identity by $s$ is enumerated by the Catalan numbers. Another way of viewing this stack sorting map is by viewing the stack as being required to avoid the permutation $21$ reading its contents from top to bottom. 

Recently Cerbai, Claesson, and Ferrari \cite{Cerbai} studied the properties of a map defined by sending a permutation through a stack which avoids a single permutation, then sending it through the stack sort map. Then, Baril, Cerbai, Khalil, and Vajnovszki \cite{Baril} studied a similar map, where instead of the first stack avoiding one permutation, it avoids two. They proved an enumeration of the set of permutations which are sorted to the identity by this map for two particular pairs of permutations. In a recent paper, Cerbai also explored the generalization to Cayley permutations \cite{CerbaiMachines}. Additionally, Cerbai et. al. prove more results about the $132$-avoiding stack in particular \cite{Cerbai132}. The vast majority of the literature has taken the approach of, for a particular stack sort style map, investigating which permutations are sorted to the identity. 

In this paper, we study a particular class of these stack sort generalizations, not as a means to sort permutations to the identity, but as functions in and of themselves. This has been done for the classical stack sort map, (see \cite{Images,claesson,ColinPreims,Fertility2,Fertility3,asymptotic}). In this paper, we study the more dynamical properties of these maps, first looking into the most straightforward question: Is the map bijective? We also classify the maximum number of preimages under the map, and classify the periodic points for a particular map.

In order to define our generalization, we first need some definitions.  The generalization we will be studying is the map $s_T$. 

\begin{definition}
Let $T$ be a set of permutations. The map $s_T$ sorts permutations (or words) according to the following algorithm: If adding the next element of the input to the stack keeps the stack $T$-avoiding, then move that element onto the stack. Otherwise, move the top element off the stack and append it to the output.
\end{definition}

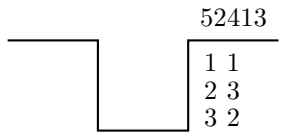
\begin{figure}[h]\begin{center}
\begin{tikzpicture}[scale=0.6]
				\draw[thick]  (4,0) -- (6,0) -- (6,-2) -- (8,-2) -- (8,0) -- (10,0);
				\node[fill = white, draw = white] at (9,.5) {52413};
			    \node[fill = white, draw = white] at (8.5,-1.7) {3};
			    \node[fill = white, draw = white] at (8.5,-1.1) {2};
			    \node[fill = white, draw = white] at (8.5,-.5) {1};
			    \node[fill = white, draw = white] at (9,-1.7) {2};
			    \node[fill = white, draw = white] at (9,-1.1) {3};
			    \node[fill = white, draw = white] at (9,-.5) {1};
              \end{tikzpicture}

\end{center} 
\caption{The map $s_T$ for $T = \{123,132\}$ before it sorts the permutation $52413$.}

\end{figure}
One might wonder if there are distinct sets of permutations $T\neq R$ such that $s_T = s_R$. 

\begin{definition}
Let $T$ be a set of permutations. We say that $T$ is \textit{reduced} if and only if there are no $\sigma,\tau\in T$ such that $\sigma$ contains $\tau$. 
\end{definition}

For distinct sets of reduced permutations $T$ and $R$ we have that $s_T$ and $s_R$ are distinct. To see this consider a permutation $\gamma$ of minimum length $n$ in $T\setminus R$. Then $s_T(\gamma^r) = \gamma(2)\gamma(1)\gamma(3)\cdots \gamma(n)$. In order for $s_R(\gamma^r)=\gamma(2)\gamma(1)\gamma(3)\cdots \gamma(n)$ to hold, $R$ would need to contain a permutation $\pi$ which is contained in $\gamma$. This contradicts the definition of reduced and so $s_t \neq s_R$.   It is also straightforward that for any set of permutations $T'$ there exists a reduced set $T$ such that $s_{T'} = s_{T}$. Thus, it is sufficient to only consider $s_T$ for reduced sets $T$.

In \cref{sec:bijSec}, we prove the following theorem, thereby classifying for which $T$ the map $s_T$ is a bijection. For any length $k$ permutation $\sigma$, let $\sigma(i)$ denote the $i$th entry of $\sigma$. Following \cite{Cerbai}, we let $\hat{\sigma}$ denote $\sigma(2)\sigma(1)\sigma(3)\cdots \sigma(k)$.

\begin{restatable}{theorem}{bijection}
\label{thm:bijection}
Let $T$ be a reduced set of permutations. The map $s_{T}$ is bijective if and only if for every $\sigma \in T$ we also have $\hat{\sigma}\in T$. In this case, the map $r \circ s_{T} \circ r$ is its inverse, where $r$ is the reverse operation.\end{restatable}

A corollary of \cref{thm:bijection} answers a question of Baril, Cerbai, Khalil, and Vajnovszki \cite{Baril} about the $(\sigma,\tau)$-machine. 

\begin{definition}
Let $\sigma$ and $\tau$ be length $n$ permutations. The $(\sigma,\tau)$-machine is the map $s \circ s_{\{\sigma,\tau\}}$.
\end{definition}
 
\begin{figure}[h]\begin{center}
\begin{tikzpicture}[scale=0.6]
                \draw[->] (8,1) -- (2,1); 
				\draw[thick] (0,0) -- (2,0) -- (2,-2) -- (4,-2) -- (4,0) -- (6,0) -- (6,-2) -- (8,-2) -- (8,0) -- (10,0);
				\node[fill = white, draw = white] at (9,.5) {52413};
			    \node[fill = white, draw = white] at (4.5,-1.5) {1};
			    \node[fill = white, draw = white] at (4.5,-.8) {2};
			    \node[fill = white, draw = white] at (8.5,-1.7) {2};
			    \node[fill = white, draw = white] at (8.5,-1.1) {3};
			    \node[fill = white, draw = white] at (8.5,-.5) {1};
			    \node[fill = white, draw = white] at (9,-1.7) {2};
			    \node[fill = white, draw = white] at (9,-1.1) {1};
			    \node[fill = white, draw = white] at (9,-.5) {3};
              \end{tikzpicture}

\end{center}
\caption{The $(132,312)$-machine about to sort the permutation 52413.}
\end{figure}
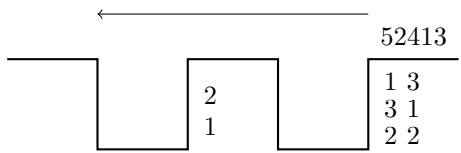

  Cerbai, Claesson, and Ferrari \cite{Cerbai} studied a version of this map $s \circ s_{\{\sigma\}}$ for various $\sigma$. The $(\sigma,\tau)$-machine was explored by Baril, Cerbai, Khalil, and Vajnovszki \cite{Baril}. In their paper, they categorize which permutations are sorted to the identity by the $(123,132)$-machine and the $(132,231)$-machine.

We denote the set of permutations of length $n$ sorted to the identity by the $(\sigma,\tau)$-machine by $\sort_n(\sigma,\tau)$. \cref{thm:bijection} results in the following corollary:

\begin{restatable}{cor}{machineCorollary}
Let $\sigma$ be a permutation. Then $|\sort_n(\sigma,\hat{\sigma})| = C_n$ where $C_n$ is the $n$th Catalan number.\end{restatable}

In particular, $\sort(132,312)$, $\sort(231,321)$ and $\sort(123,213)$ are enumerated by the Catalan numbers. This answers a question of Baril, Cerbai, Khalil, and Vajnovszki \cite{Baril}. Along with Baril, Cerbai, Khalil, and Vajnovszki's result that $\sort(132,123)$ is enumerated by the Catalan numbers, our result fully classifies for which pairs $\sigma,\tau$ of length 3 permutations we have $\sort_n(\sigma,\tau) = C_n$ 

An alternative proof of \cref{machineCorollary} was given at Permutation Patterns 2020 by Cerbai. 

 Properties of the \textit{fertility} of a permutation, or the size of its preimage under the stack sort map have been explored in various papers \cite{Fertility,Fertility2,Fertility3,ColinLassalle}. There is a strong link between fertilities of permutations and cumulants in noncommutative probability theory \cite{ColinLassalle, asymptotic}.  Problems regarding fertilities of permutations have been well studied in the stack sorting case. It is a classical result that the maximum fertility of a length $n$ permutation is simply $C_n$. In \cref{sec:preims}, we generalize this result to $s_T$, proving the following.

\begin{theorem}
If $T$ is a set of permutations, all of length at least $k$, then every permutation of length $n>k$ has at most $C_{n-k+2}$ preimages under the map $s_{T}$.

\end{theorem}

We also prove for the case when $T$ is of size one exactly when this bound is attained. 

\begin{theorem}
Let $\sigma$ be a length $k$ permutation. The maximum number of preimages a length $n$ permutation has under $s_{\{\sigma\}}$ is $C_{n-k+2}$ if and only if $\sigma(1)$ and $\sigma(2)$ are consecutive integers.
\end{theorem}

In \cref{sec:periodicity}, we look into the periodic points of the map $s_T$ for the case when $T = \{213,\tau\}$ for non-identity $\tau$ of length at least 3.

\section{Bijectivity of \texorpdfstring{$s_T$}{map} and the \texorpdfstring{$(\sigma,\tau)$}{(sigma,tau)}-machine} \label{sec:len3}\label{sec:bijSec}

The goal of this section is to prove \cref{thm:bijection}.

For any $\pi\in S_n$ we will let $\pi^r$ denote $\pi(n) \cdots \pi(1)$. Let $r:S_n\to S_n$ be such that $r(\pi) = \pi^r$. Let $T^r$ denote the set $\{\sigma^r: \sigma \in T\}$

\begin{definition}
For permutations $\sigma$ and $\tau$ we say that $\sigma<\tau$ \textit{colexicographically} if $\sigma^r>\tau^r$ lexicographically. 
\end{definition}

\begin{definition}
Let $T$ be a set of permutations and $\pi$ a permutation of length $n$ such that $\pi^r$ is not $T$-avoiding. We call a tuple $(a_0,\dots, a_{k})$,  of contiguous subsequences of $\pi$, the \textit{$T$-clumping} of $\pi$ if the following holds: \begin{enumerate}
    \item $a_0\cdots a_{k} = \pi$
    \item The colexicographically least element of the set $$\{(i_1,\dots,i_k): 1\leq i_1 < \cdots < i_k\leq n \textrm{ and } \pi(i_1) \cdots \pi(i_k) \cong \sigma^r \textrm{ for some }\sigma \in T\}$$
    is the tuple of indices of $a_1(1),\dots, a_k(1)$ in $\pi$. 
\end{enumerate}   
In other words, a $T$-clumping is the decomposition of a permutation $\pi$ induced by the indices of the leftmost occurrence of a pattern $\sigma \in T^r$ in $\pi$, with respect to the colexicographical order of indices.
\end{definition}

Note that any permutation which is not $T^r$-avoiding has a unique $T$-clumping. For instance if $T = \{123,132\}$ Then the $T$-clumping of $731426$ is $a_0 = 73$, $a_1 = 1$, $a_2 = 4$, $a_3 = 26$.

\begin{theorem} \label{thm:recursion}
Let $T$ be a set of permutations. Then $s_{T}$ satisfies the following recurrence:

If $\pi$ is a permutation or word avoiding $T^r$, then $s_{T}(\pi) = \pi^r$. 

Otherwise, let $(a_0,\dots, a_{k})$ be the $T$-clumping of $\pi$. Then $$s_{T}(\pi) = s_{T}(a_0\cdots a_{k}) = a_{k-1}^r s_{T}(a_0\cdots a_{k-2} a_{k}).$$
\end{theorem}

\begin{proof}

Assume $\pi$ does not avoid $T$. If $\pi = a_0\cdots a_{k}$ where $(a_0,\dots, a_{k})$ is a $T$-clumping, then when $a_0\cdots a_{k}$ is sorted by $s_{T}$, first $a_0\cdots a_{k-1}$ enters the stack since $a_0\cdots a_{k-1}$ is $T^r$-avoiding, by definition of a $T$-clumping. Then since $a_{1}(1)\cdots a_{k}(1)$ forms a $\sigma^r$ for some $\sigma \in T$, before $a_{k}(1)$ can enter the stack, first $a_{k-1}$ must leave the stack. Thus, $a_{k-1}^r$ gets appended to the (currently empty) output and $a_0,\dots,a_{k-2}$ is still in the stack, while $a_{k}$ is still part of the input. Now, consider the situation in which we are inputting $a_0\cdots a_{k-2}a_{k}$ into $s_{T}$. Since we know that $a_0\cdots a_{k-2}$ is $T^r$-avoiding, $a_0\cdots a_{k-2}$ enters the stack. Here we are in the same situation as previously, but without the $a_{k-1}^r$ appended to the output. Thus, in the situation of inputting $a_0\cdots a_{k}$, we end up with $a_{k-1}^r$ followed by the permutation obtained from applying $s_T$ to $a_0\cdots a_{k-2}a_{k}$. Thus, $s_{T}$ satisfies the recurrence. 
\end{proof}

Recall that for a permutation $\sigma$ of length $k$, we let $\hat{\sigma}$ denote the permutation $\sigma(2)\sigma(1)\sigma(3)\cdots \sigma(k)$.

\bijection*

For convenience, we let $R_{T} =r \circ s_{T} \circ r $. 

\begin{proof}
Assume that $\hat{\sigma}\in T$ for every $\sigma\in T$. We show that $R_T$ is the inverse of $s_T$. We will proceed by induction on $n$, the length of our permutation or word $\pi$.

Let $m$ be the minimum length of a permutation in $T$. For our base case, we let $\pi$ be a word of length $m-1$ or fewer. Then $\pi$ must avoid $T^r$ so $s_{T}\circ R_{T} (\pi)=s_{T}(\pi^r) = \pi $. 

Assume that for all permutations or words $\rho$ of length less than $n$,  $s_{T}\circ R_{T} (\rho)= \rho $. Let $\pi$ be a permutation or word of length $n$. If $\pi$ is $T^r$-avoiding then $s_{T}\circ R_{T} (\pi)=s_{T}(\pi^r) = \pi $. Assume $\pi$ is not $T^r$-avoiding.  Let $(a_0,\dots,a_{k})$ be the $T$-clumping of $\pi$. Let $\sigma \in T$ be so that $a_1(1)\cdots a_{k}(1)$ is an occurrence of $\sigma^r$.  Then
\begin{align*}
    R_T\circ s_T(\pi)  &= R_T\circ s_T(a_0\cdots a_{k})\\
    &= R_T( a_{k-1}^rs_T(a_0\cdots a_{k-2} a_{k}))\\
    & =r \circ s_T \circ r( a_{k-1}^rs_T(a_0\cdots a_{k-2} a_{k}))\\
    &  = r \circ s_T \left[ (s_T(a_0\cdots a_{k-2}a_{k}))^ra_{k-1}\right]\\
   & =r \circ s_T ( s_T^{-1}(a_{k}^r a_{k-2}^r\cdots a_0^r)a_{k-1}).
\end{align*}

Note that it makes sense to write $s^{-1}_T$ since the map is guaranteed to be bijective on words of length less than $n$ by our induction hypothesis. The last equality is by our induction hypothesis.

Now let's restrict our attention to $s_{T} ( s_{T}^{-1}(a_{k}^r a_{k-2}^r\cdots a_0^r)a_{k-1})$. Note that by the definition of a $T$-clumping, $a_0,\dots,a_{k-2}a_{k}(1)$ must avoid $T^r$. Thus, in order to be sorted to $a_{k}^r a_{k-2}^r\cdots a_0^r$ by $s_T$, when $s_{T}^{-1}(a_{k}^r a_{k-2}^r\cdots a_0^r)$ is sorted,  $a_0,\dots,a_{k-2}a_{k}(1)$ must be at the bottom of the stack when there are no more elements to enter the stack. 

To see this, assume for the sake of contradiction that we have a permutation $\alpha$ such that $s_T(\alpha)$ ends with $a_{k}^r a_{k-2}^r\cdots a_0^r$, but in the process of sending $\alpha$, through the stack, we don't have $a_0,\dots,a_{k-2}a_{k}(1)$ in the bottom of the stack when there are no more elements to enter the stack. This would mean $\alpha$ had to end with some $\gamma\neq a_0,\dots,a_{k-2}a_{k}(1)$ which gets sorted to $a_{k}^r a_{k-2}^r\cdots a_0^r$ by $s_T$. However, since $s_T$ couldn't simply reverse $\gamma$ we would have that $\gamma$ contains some $\sigma \in T^r$. But then, $s_T(\gamma)$ would contain an occurrence of $\hat{\sigma}$. Since $\hat{\sigma}\in T$, this contradicts the definition of a $T$-clumping. 

Thus, when inputting 
$s_{T}^{-1}(a_{k}^r a_{k-2}^r\cdots a_0^r)a_{k-1}$ into the map $s_T$, at the point that $a_{k-1}(1)$ tries to enter the stack, we have that an initial segment of $a_{k}^r$ is appended to the output, and the rest of $a_{k}^r$ is at the top of the stack followed by $a_{k-2}^r,\dots,a_0^r$ at the bottom. Since $a_{2}(1)\cdots a_k(1) $ is an occurrence of $\sigma^r$, $a_2(1),\dots a_{k-2}(1) a_{k}(1) a_{k-1}(1)$ is an occurrence of $\hat{\sigma}^r$, so $a_k(1)$ and everything above it must leave the stack before $a_{k-1}(1)$ can enter. Because $a_0\cdots a_{k-1}$ must be $T$-avoiding, the rest of $a_{k-1}$ can be added to the stack. Once this happens we have $a_k^r$ as part of the output and $a_{k-1}^r \cdots a_0^r$ in the stack. Since there is nothing left unsorted, the stack pops out to form $a_k^r \cdots a_0^r$.

Thus, we have that $$r \circ s_T ( s_T^{-1}(a_k^r a_{k-2}^r\cdots  a_0^r)a_{k-1})  = r (a_k^r \cdots a_1^r) = a_0\cdots a_k = \pi.$$ 

\noindent So, $R_T = s_T^{-1}$ and $s_T$ is bijective. 

For the converse, let $\sigma\in T$ be a permutation of length $k$ such that $\hat{\sigma}\not \in T$. Consider the permutations $\sigma^r$ and $(\hat{\sigma})^r$. When $\sigma^r$ gets sorted, we have that $\sigma(k),\dots,\sigma(2)$ can enter the stack while keeping the stack $T$-avoiding, since $T$ is reduced. However, $\sigma(2)$ must leave the stack before $\sigma(1)$ enters. Thus, the permutation gets sorted to $\sigma(2)\sigma(1)\sigma(3)\cdots \sigma(k) = \hat{\sigma}$.  Since $\hat{\sigma}\not \in T$,  we have that $s_{T}((\hat{\sigma})^r) = \hat{\sigma}$. Thus the map $s_T$ is not bijective. 
\end{proof}

Recall that for a permutation $\sigma$, the $(\sigma,\hat{\sigma})$-machine is equal to s $\circ s_{\{\sigma,\hat{\sigma}\}}$ and that $\sort_n(\sigma,\hat{\sigma})$ is the preimage of the length $n$ identity under that map. 

\cref{thm:bijection} results in the following corollary.

\machineCorollary*

\begin{proof}
The $(\sigma,\hat{\sigma})$-machine is equal to  $s\circ s_{\{\sigma,\hat{\sigma}\}}$. It was proved by Knuth that $|{s}^{-1}(\id_n)| = C_n$ in \cite{Knuth}. Here we write $\id_n$ to denote the identity permutation of $S_n$. Thus, since $s_{\{\sigma,\hat{\sigma}\}}$ is bijective, we have  
$$|s_{\{\sigma,\hat{\sigma}\}}^{-1}( s^{-1}(\id_n)) | = |s^{-1}(\id_n)| = C_n .\qedhere$$ \end{proof}
The following table gives the first few values for the cardinality of  sort$(\sigma,\tau)$ for each possible pair  $(\sigma,\tau)$ of length three. The sequences which are equal to the Catalan sequence are underlined. 

\begin{center}
\begin{tabular}{ |c|c||c|c|| c| c| } 
 \hline
\underline{$(123,132)$} & {1  2  5  14} & $(123,321)$ & 1  2  4  7  & $(213,231)$ & 1 2 5 16\\
 \underline{$(123,213)$} & {1  2  5  14} & $(132,213)$ & 1 2 5 15  & $(213,321)$& 1 2 4 12\\
  $(123,231)$ & 1  2  6  21  & \underline{$(132,312)$} & {1  2  5  14} &$(231,312)$& 1 2 6 23 \\
 $(123,231)$ & 1  2  6  21 & $(132,321)$ & 1 2 4 10 &\underline{$(231,321)$} & {1  2  5  14}\\
  $(123,312)$ & 1  2  5  15  & $(213,231)$ & 1 2 6 23& $(312,321)$& 1 2 4 10\\
 \hline
\end{tabular}
\end{center}

Thus, we may observe that for all length three patterns $(\sigma,\tau)$ other than $(123,132)$, $(123,213)$, $(132,312)$, and $(231,321)$, sort$(\sigma,\tau)$ is not counted by the Catalan numbers. Baril, Cerbai, Khalil, and Vajnovszki \cite{Baril}, proved that sort$(123,132)$ is counted by the Catalan numbers. The other three cases are proved by Corollary 2.1. Thus, we now have a full classification of the pairs of length 3 permutations $\sigma,\tau$ for which sort$(\sigma,\tau)$ is  enumerated by the Catalan numbers.

\section{ Preimages of the map \texorpdfstring{$s_{\{\sigma\}}$}{(tau)} } \label{sec:preims}

As in the solution of exercise 19 from Chapter 2 of \cite{Bona}, we will define the notion of \textit{movement sequences}, but now generalized to $s_{T}$.

Note that at any point in the process of sorting using a map $s_T$, two things can happen: the next entry can enter the stack, or the top element of the stack can exit the stack. 
\begin{definition}\label{movement-sequence}
The \textit{movement sequence} of a permutation under a stack sorting map $s_{T}$ is the sequence of steps the permutation follows as it gets sorted. We will use $N$ to denote the enter step, and $X$ to denote the exit step. 

Recall the notion of a Dyck word. A word $w$ from the alphabet $\{U,D\}$ is a \textit{Dyck word} if for any $i \leq \textrm{len} (w)$ we have that $U$ appears more times than $D$ in the substring $w(1)\cdots w(i)$. 

The standard visualization for Dyck words is as a Dyck path. Given a word $w$ on the alphabet $\{U,D\}$, to obtain a path one starts at the origin and at step $i$ moves diagonally up and left if $w(i) = u$ and diagonally down and left if $w(i)$. Dyck words correspond to Dyck paths which do not travel lower than the starting point.
 
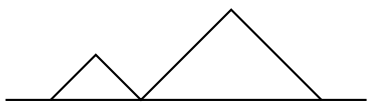
\begin{figure}[h]\begin{center}
\begin{tikzpicture}[scale=0.6]
				\draw[thick]  (0,0) -- (1,1) -- (2,0) -- (4,2) --  (6,0);
				\draw[thick]  (-1,0) --   (7,0);
              \end{tikzpicture}
\end{center} 
\caption{The Dyck path corresponding to $UDUUDD$.}

\end{figure}
\end{definition}

\begin{theorem}\label{bound}
If $T$ is a set of permutations, all of length at least $k$, then every permutation of length $n>k$ has at most $C_{n-k+2}$ preimages under the map $s_{T}$.
\end{theorem}

\begin{proof}
Note that any point in the movement sequence of any permutation there must be at least as many entries that have entered the stack as entries that have left the stack. Thus there is a clear bijection between potential movement sequences of length $n$ permutations and Dyck paths of semi-length $n$. 

If all elements of $T$ are  of length at least $k$, then while sending a permutation through $s_{T}$, the first $k-2$ entries of the permutation stay at the bottom of the stack for the entire process. Thus, given a length $n$ permutation $\pi$, the movement sequence of $\pi$ under the map $s_{T}$ begins with $k-2$ ``enter"s and ends with $k-2$ ``exit"s. Therefore, its corresponding Dyck path doesn't pass below the $k-2$ height line at any point in between the beginning and the end. Such Dyck paths are in bijection with Dyck paths of semi-length $n-k+2$ by removing these predetermined steps. Thus, there are at most $C_{n-k+2}$ movement sequences for a permutation under the map $s_{T}$.

 If $\pi$ and $\rho$ are permutations such that $s_T(\pi) = s_T(\rho)$ where $\pi$ and $\rho$ share the same movement sequence then it must be the case that $\pi = \rho$. This is because, given a movement sequence and a permutation in the image of $s_{T}$, it is possible to reverse the steps to obtain the original permutation. Thus, for any $\gamma$ in the image of $s_{T}$, all distinct permutations mapping to $\gamma$ must have distinct movement sequences. This means there can only be as many preimages under the map $s_{T}$ for a given element of the image as there are movement sequences for the map.  
\end{proof}

\begin{definition}
Let $T$ be a set of permutations and let $k$ be the minimum length of a permutation in $T$. Let $N$ and $X$ denote the two possible entries of the movement sequence, ``enter the stack" and ``exit the stack." We will say that a movement sequence is a $T$\textit{-legal movement sequence} if it is of the form $N^{k-2}m X^{k-2}$ where $m$ is a movement sequence. 
\end{definition}

Next, we classify for which $s_{\{\sigma\}}$ the bound stated in \cref{bound} is sharp. In order to prove the next theorem we will need some definitions.

\begin{definition}
Let $\pi$ be a length $n-1$ permutation. Then the permutation $$(n-\pi(1))(n-\pi(2))\cdots (n-\pi(n-1))$$ is the \textit{complement} of $\pi$.
\end{definition}

For example, the complement of $23145$ is $43521$. We let $\sigma^c$ denote the complement of $\sigma$ and $T^c$ denote the set $\{\sigma^c: \sigma\in T\}$.

\begin{lemma}\label{lem: complement}
Let $\pi$ be a permutation of length $n$ and $T$ be a reduced set of permutations. If $s_T(\pi) = \rho$ then $s_{T^c}(\pi^c) = \rho^c$. 
\end{lemma}

\begin{proof}
Let $\pi$ and $T$ be as in the lemma. We will prove this by using the recursive definition in \cref{thm:recursion}. We will induct on $n$. If $n<m$ where $m$ is the minimum length of a permutation in $T$, then $\pi$ is $T$-avoiding. Therefore $\pi^c$ is $T^c$-avoiding so $$s_{T^c}(\pi^c) = (\pi^c)^r = (\pi^r)^c = (s_T(\pi))^c.$$ Assume for all $\ell \leq n$, for all permutations of length $\ell$ the lemma holds.  

If $\pi$ is $T$-avoiding, then $\pi^c$ is $T^c$-avoiding so $$s_{T^c}(\pi^c) = (\pi^c)^r = (\pi^r)^c = (s_T(\pi))^c.$$ Assume $\pi$ contains an element of $T$. Let $(a_0,\dots,a_k)$ be the $T$-clumping of $\pi$. Then $(a_0^c,\dots,a_k^c)$ must be the $T^c$-clumping of $\pi^c$. Thus, $$s_{T^c}(\pi^c) = (a_{k-1}^c)^r s_{T^c}(a_0^c,\dots,a_{k-2}^c,a_k^c)  = (a_{k-1}^r s_{T}(a_0,\dots,a_{k-2},a_k))^c  =  (s_{T}(\pi))^c.$$

The second equality is due to our induction hypothesis. Thus, by induction, our lemma holds. 
\end{proof}

\begin{definition}
Let $\sigma$ be a word of length $k$ in the alphabet $\mathbb{Z}\cup \{m^c: m\in \mathbb{Z}\}$. A length $n$ permutation $\pi$ \textit{literally contains} $\sigma$ if there are indices $i_1<\cdots <i_k$ such that for each $j\leq k$:
$$\begin{cases}
\pi^c(i_j) = m, \hspace{2mm} \textrm{ if }\sigma(j) = m^c\\
\pi(i_j) = \sigma(j),\hspace{1mm} \textrm{ otherwise}
\end{cases} $$ Where $\pi^c$ is the complement of $\pi$. 
\end{definition}

For example, the permutation $1423$ \textit{literally contains} the permutation $11^c 2$ in the subsequence $142$. The permutation $1243$ does not literally contain $11^c 2$ because though it contains $143$, the three is not equal to two in the literal sense.

\begin{theorem}\label{generalizationPreims}
Let $\sigma$ be a length $k$ permutation. If $\sigma(1)$ and $\sigma(2)$ are consecutive numbers, then for every $n \geq k$, there exists a permutation $\pi\in S_n$ such that $|s_{\{
\sigma\}}^{-1}(\pi)| = C_{n-k+2}$. If $\sigma(1)$ and $\sigma(2)$ are not consecutive, then for every $n>k$ there are no $\pi\in S_n$ such that $|s_{\{
\sigma\}}^{-1}(\pi)| = C_{n-k+2}$.
\end{theorem}

\begin{proof}
Assume $\sigma$ is a length $k$ permutation such that $\sigma(1)$ and $\sigma(2)$ are consecutive. Let $\kappa(\sigma)$ be a length $k-2$ word over the alphabet $\Z\cup \{m^c: m \in \Z\}$ defined as follows: $$\kappa(\sigma)(i) =  \begin{cases}
\sigma(i+2) & \textbf{ if } \sigma(i+2)< \sigma(1)\\
(k+1-\sigma(i+2))^c & \textbf{ if } \sigma(i+2)> \sigma(1).
\end{cases}$$ 
Note that for any permutation $\pi$ of length $n<k$ we have that $s_{\{\sigma\}}$ simply reverses $\pi$, since $\pi$ cannot contain an occurrence of $\sigma$. Fix $n\geq k$. Let $\mu(\sigma)$ be the length $n$ permutation whose last $k-2$ entries are an occurrence of  $\kappa(\sigma)$, and the remaining $n-k+2$ elements are increasing if $\sigma(1)>\sigma(2)$ or decreasing if $\sigma(1)<\sigma(2)$. Let $P_{\sigma,n}$ be the set of length $n$ permutations whose first $k-2$ entries form an occurrence of $ \kappa(\sigma)^r$ and the remaining $n-k+2$ entries form a $231$-avoiding permutation if $\sigma(1)>\sigma(2)$ or a $213$-avoiding permutation if $\sigma(1)<\sigma(2)$. Our goal is to show that $s_T(\mu(\sigma)) = P_{\sigma, n}$ which has size $C_{n-k+2}$.

Let's consider the case where $\sigma(1)>\sigma(2)$ and examine what happens an arbitrary element of $\pi\in P_{\sigma,n}$ is sent through $s_T$. First, $\kappa(\sigma)^r$ enters the stack to form a $\kappa(\sigma)$ in the stack. This occurrence of $\kappa(\sigma)$ stays in the stack until the end. By our construction of $\kappa(\sigma)$, any decrease which appears in the stack after $\kappa(\sigma)$ causes an occurrence of $\sigma$. Thus, with $\kappa(\sigma)$ at the bottom of the stack, the stack acts as a $21$-avoiding stack and the rest of the permutation is sorted accordingly. Since $\rho$ is 231-avoiding, and 231-avoiding permutations are exactly those mapped to the identity by the $21$-avoiding stack $s$ \cite{Knuth}, the remaining elements of the stack are correctly sorted increasingly. The word $\kappa(\sigma)$ exits afterward to form the permutation $\mu(\sigma)$. 

For the case where $\sigma(1)< \sigma(2)$, \cref{lem: complement} gives us the same result. To see this note that if $\sigma(1)<\sigma(2)$ then $\sigma^c(1)>\sigma^c(2)$, $\mu(\sigma^c) = \mu(\sigma)^c$, and $P_{\sigma^c,n} = P_{\sigma,n}^c$. So, we have by the first case that $s_{\sigma_c}(\mu(\sigma^c)) = P_{\sigma^c, n} = P_{\sigma,c}^c$ and so by \cref{lem: complement}, $s_{\sigma}(\mu(\sigma)) = P_{\sigma,n}$. 

Since the number of 231-avoiding permutations is counted by the Catalan numbers, the size of $P_{\sigma,n}$ is counted by $C_{n-k+2}$ where $k$ is the length of $\sigma$. Thus, the permutation $\mu(\sigma)$ has preimage of size $C_{n-k+2}$.

For the converse direction, assume $\sigma$ is a length $k$ permutation such that $\sigma(1)$ and $\sigma(2)$ are not consecutive. We will show that for each $n$ every permutation in $S_n$ has fewer than $C_{n-k+2}$ many preimages. Consider the movement sequence  $N^{k-2}(NX)^{n-k+2}X^{k-2}$. For example, if $n$ were 5 and $k$ were  $4$, this sequence would be $NNNXNXNXXX$. Call this movement sequence $m$. 

Note that in order for movement sequence $m$ to occur for a permutation $\pi$ of length $n>k$, we must have that once the first $k-1$ elements of $\pi$ have entered the stack, every entry after would form an occurrence of $\sigma$ in the stack.  In order for this to happen, it must be true that for every $i$ such that  $k-2<i\leq n-1$ we have that $\pi(i+1)\pi(i)\pi(k-2) \cdots \pi(1)$ is order isomorphic to $\sigma$. Assume this is true. Since $\sigma(1)$ and $\sigma(2)$ are not consecutive, let $j\leq k$ be so that $\sigma(1)<\sigma(j)<\sigma(2)$, without loss of generality. Then note that $\pi(k+1-j)$ is the element of $\pi$ which maps to $\sigma(j)$ in the order isomorphism between $\pi(i+1)\pi(i)\pi(k-2) \cdots \pi(1)$ and $\sigma$, for any $i<n-1$.  Then, since $\pi(k)\pi(k-1)\pi(k-2) \cdots \pi(1)$ is an occurrence of $\sigma$, it must be true that $\pi(k)<\pi(k+1-j)$. However, since $\pi(k+1)\pi(k)\pi(k-2) \cdots \pi(1)$ is an occurrence of $\sigma$, it must be true that $\pi(k)>\pi(k+1-j)$. This is a contradiction. Thus, if $\sigma(1)$ and $\sigma(2)$ are not consecutive, no permutation of length $n>k$ has movement sequence $m$. So, since there are only $C_{n-k+2}$ many $\{\sigma\}$-legal movement sequences of length $2n$ by \cref{bound}, and we ruled out one possible option, there are fewer than $C_{n-k+2}$ many movement sequences left. Since for a given permutation, every element of its preimage under $s_{\{\sigma\}}$ must have been sorted by a different movement sequence, we then have that the size of the preimage is constrained to be of size less than $C_{n-k+2}$ for any $n$-length permutation. 
\end{proof}

We will now further explore the maximum number of preimages under the map $s_{\{213,231\}}$.

\begin{theorem}\label{213taustack}
Let $\tau$ be a non-identity permutation of length at least three. Then the number of permutations of length $n$ sorted to the identity by $s_{\{213,\tau\}}$ is counted by $C_{n-1}$. Additionally, the permutations which sort to the identity are exactly those of the form $n\rho$ for $\rho \in S_{n-1}$ avoiding 231.   
\end{theorem}

\begin{proof}
Assume that $\pi\in S_n$ begins with $n$. For any $s_{\{\sigma,\tau\}}$ for $\sigma$ and $\tau$ of length at least 3, the first element remains at the bottom of the stack until the end, when it is added to the end of the output. Thus, throughout the this process, $\pi(1) = n$ stays at the bottom of the stack. Since $n$ is greater than any other entry of $\pi$, any occurrence of $12$ in the rest of the permutation would cause an occurrence of $312$ in the permutation and thus $213$ in the stack. Therefore, with the $n$ lingering in the bottom of the stack, the rest of the permutation gets sorted with the rest of the stack avoiding $21$ as well as $\tau$. Since $\tau$ contains $21$ this is equivalent to having the stack avoid 21. Note that this is equivalent to classical stack sort. It is a well known theorem of Knuth \cite{Knuth} that the permutations sorted to the identity by the stack sort algorithm are exactly those which avoid 231. Thus, those elements starting with $n$ which get sorted to the identity are exactly those $\pi$ such that $\pi(2)\cdots \pi(n)$ avoids 231.

Note that any permutation not beginning with $n$ cannot be sorted to the identity permutation by $s_{\{213,\tau\}}$ because the first entry of the permutation is always the last entry in the output. This is easy to see since the first entry always stays at the bottom of the stack until the end. 
\end{proof}

\cref{213taustack} has the following corollaries.

\begin{cor}\label{cor:213stack}
Let $\pi\in S_n$ be a permutation. Then $s_{213}(\pi) = \operatorname{id}_n$ if and only if $\pi = n\rho$ for 231-avoiding $\rho \in S_{n-1}$.
\end{cor}

This corollary can be easily seen by letting $\tau = 213$. 
\begin{cor}\label{cor:231stack}
Let $\tau$ be a permutation of length at least 3 which is not equal to the reverse identity. Then the number of permutations of length $n$ sorted to the reverse identity by $s_{231,\tau}$ is counted by $C_{n-1}$. Additionally, the permutations which sort to the reverse identity are exactly those of the form $1\rho$ for $\rho \in S_{n-1}$ avoiding 213.   
\end{cor}

This corollary is a result of the \cref{lem: complement}.

Next we will show that the identity and the reverse identity are the only two permutations with maximal sized preimage sets under the map $s_{\{213,231\}}$.

\begin{cor}\label{213,231preims}
The maximal number of preimages for any permutation of length $n$ under the map $s_{213,231}$ is $C_{n-1}$. This maximum is attained only by $\textrm{id}_n$ and $\textrm{id}_n^r$.
\end{cor}

\begin{proof}
The fact that the maximum number of preimages of a permutation of length $n$ under $s_{213,231}$ is $C_{n-1}$ follows directly from \cref{bound}. That the identity and its reverse attain this maximum follows from \cref{213taustack} and \cref{cor:231stack}. All there is left to show is that no other permutation attains the maximum number of preimages. Let's assume $\pi$ has the maximum number of preimages. Note that by the argument in \cref{bound}, for a permutation $\pi$ to have the maximum number of preimages, for each $\{213,231\}$-legal movement sequence $v$, there must be an element of the preimage $\rho$ such that by following $v$, $\rho$ is sorted to $\pi$. Then, there must be some permutation $\rho$ such that $N(NX)^{n-1}X$ sorts $\rho$ to $\pi$. This is only possible if \textit{every} time a third entry tries to enter the stack, it causes a 213 or 231 in the stack. The only way for this to be possible is if $\rho$ is of the form $n12\cdots (n-1)$ or $1n (n-1) \cdots 2$ . In the first case, $\rho$ sorts to $\textrm{id}_n$ and in the second, $\rho$ sorts to $\textrm{id}_n^r$. Thus, the movement sequence $N(NX)^{n-1}X$ must only be used by $s_{213,231}$ to sort permutations to the identity and its inverse. Thus, no other permutation has the maximum number of preimages. 
\end{proof}

\section{ Periodic points of \texorpdfstring{$s_{\{123,132\}}$}{s (sigma, tau)}} \label{sec:periodicity}

The goal of this section is to classify the permutations which are periodic points under the map $s_{\{123,132\}}$. We will begin with some definitions.

\begin{definition}
Given a map $f:A \to B$, an element $a\in A$ is a \textit{periodic point} if for some $n\in \N$ (with $n>0$) we have that $f^n(a) = a$. 
\end{definition}

\begin{definition}
Let $\pi$ be a permutation of length $n$. We say that $\pi$ is \textit{half-decreasing} if the subsequence $$\pi(n-1)\pi(n-3)\cdots \pi(2) \textrm{ for odd }n$$
$$\pi(n-1)\pi(n-3)\cdots \pi(3) \textrm{ for even }n$$
is the identity of length $\lfloor \frac{n-1}{2}\rfloor$. (Being order isomorphic to the identity is not sufficient.) We will refer to this subsequence as its \textit{decreasing half}. 
\end{definition}

For example $56342718$ and $947382615$ are both half-decreasing permutations while $789342615$ and $634251$ are not.

We will next prove the following lemma:

\begin{lemma} \label{lem: half-decreasing}
If $\pi$ is a half-decreasing permutation of length $n$ then the map $s_{\{123,132\}}$ acts on it as follows:
$$s_{\{123,132\}}(\pi) =  \pi(3)\pi(2)\pi(5)\cdots \pi(n)\pi(n-1)\pi(1) \textrm{ for odd }n$$
$$s_{\{123,132\}}(\pi) =  \pi(2)\pi(4)\pi(3)\pi(6)\cdots \pi(n)\pi(n-1)\pi(1) \textrm{ for even }n.$$
In other words, the decreasing half is fixed under $s_{\{123,132\}}$ and the remaining elements shift cyclically to the left. 
\end{lemma}

\begin{proof}
We will proceed by induction on $n$. For our first base case, let $n = 3$. Since both half-decreasing permutations of length three (213 and 312) avoid 321 and 231, we have that $s_{\{123,132\}}(213) = 312$ and $s_{\{123,132\}}(312) = 213$. Thus the claim holds for $n=3$. 


For our second base case, let $n=4$. Then $\pi = \pi(1)\pi(2)1\pi(4)$. Since $\pi(1)$ and $\pi(2)$ are both greater than 1, we have that $\pi(1)\pi(2)1$ must form a 321 or a 231. Thus, by our recursion from \cref{thm:recursion},  we have $s_{\{123,132\}}(\pi) = \pi(2)s_{\{123,132\}}(\pi(1)1\pi(4))$. Since $\pi(1)1\pi(4)$ must avoid 321 and 231, we have that $s_{\{123,132\}}(\pi) = \pi(2)\pi(4)1\pi(1)$, thus satisfying the claim. 

Assume for our inductive hypothesis that for all permutations $\rho$ of length less than $n$, $s_{\{123,132\}}(\rho)$ satisfies the claim. Let $\pi$ be a half-decreasing permutation of length $n$. We will consider cases based on if $n$ is even or odd.

Case 1: $n$ is odd. Let $m = \frac{n-1}{2}$. In other words, $m$ is the largest entry in the decreasing half of $n$. Then $\pi = \pi(1)m\pi(3)(m-1) \cdots 1\pi(n)$. Since $\pi(1)$ and $\pi(3)$ are greater than $m$ and $m-1$, we have that $\pi(1)m(m-1)$ is the colexicographically first appearance of 321 or 231. Thus, $s_{\{123,132\}}(\pi) = \pi(3) m s_{\{123,132\}}(\pi(1) (m-1)\pi(5)\cdots \pi(n))$. Note that $\pi(1) (m-1)\pi(5)\cdots \pi(n)$ (when interpreted as a permutation) is half-decreasing. Thus, by our inductive hypothesis, all entries of its decreasing half are fixed and all other elements shift left cyclically. So, $s_{\{123,132\}}(\pi) = \pi(3)m \pi(5) \cdots \pi(n)1\pi(1)$ as desired. 

Case 2: $n$ is even. Let $m = \frac{n}{2}-1$. Again, $m$ is the largest entry in the decreasing half of $n$ and $\pi = \pi(1)\pi(2)m \pi(4) (m-1) \cdots 1\pi(n)$. Since $\pi(1)$ and $\pi(2)$ are greater than $m$, $\pi(1)\pi(2)m$ is the colexicographically first appearance of 321 or 231. Thus, $s_{\{123,132\}}(\pi)= \pi(2)s_{\{123,132\}}(\pi(1)m \pi(4)(m-1) \pi(6)\cdots \pi(n))$. Note that $(\pi(1)m \pi(4)(m-1) \pi(6)\cdots \pi(n)$ forms a half-decreasing sequence, and thus by our induction hypothesis $s_{\{123,132\}}(\pi) = \pi(2)\pi(4)m \pi(6)\cdots \pi(n)1\pi(1)$, as desired. This concludes our induction. 
\end{proof}

\begin{lemma} \label{lem: eventually-half-dec}
Let $\pi$ be a permutation. Then $s_{\{123,132\}}^{m}(\pi)$ is half-decreasing for some $m\in \N$. 
\end{lemma}

\begin{proof}
We will prove this lemma through a series of claims.\\ 

\noindent \textbf{Claim 1:} Let $\pi$ be a permutation of length $n\geq 3$ which ends in a half-decreasing permutation. In other words, for some $i>1$, if we let $m = \frac{n-i+1}{2}$ then we have that  $$\pi(i)\pi(i+2)\cdots \pi(n-1) = (m)(m-1)\cdots(1),$$ then these entries are fixed under $s_{\{123,132\}}$, and for all other $i<j<n$, $\pi(j)$ gets sorted to index $j-2$. \\

Let $\pi$ be as in Claim 1. While sending $\pi$ through the stack, when we get to the point that $\pi(1)\cdots\pi(i-1)$ is either sorted or in the stack, we must have at least two entries in the stack. Then, since all entries smaller than $\pi(i)$ come after index $i$, adding $\pi(i)$ to the stack would create a 123 or 132 permutation in the stack. Thus, all but one entry must leave the stack before $\pi(i)$ can enter. Note that throughout the entire process $\pi(1)$ stays at the bottom of the stack, as 123 and 132 cannot be formed by only two entries. Thus, at this point, we have $\pi(2),\dots,\pi(i)$ sorted into some order in the output, and $\pi(1)$ in the stack. At this point, $\pi(i)$ enters the stack and $\pi(i+1)$ also enters the stack, since it must be larger than $\pi(i)$.  Now, there are two possibilities: either this is the end and $\pi(i+1)$ and $\pi(i)$ pop out of the stack, leaving $\pi(i)$ at index $i$ and $\pi(i+1)$ at $i-1$, or there is more to sort. In the second case, since $\pi(i+2)$ is smaller than $\pi(1), \pi(i), \pi(i+1)$, we must have $\pi(i+1)$ and $\pi(i)$ leave the stack before $\pi(i+2)$ can enter. This would also place $\pi(i)$ at index $i$ and $\pi(i+1)$ at index $i-1$ in the output. The same reasoning applies to the other entries until the end of the permutation. Thus the claim holds.\\

\noindent \textbf{Claim 2:} Let $\pi$ be a permutation of length $n\geq 5$ such that for some $i>3$, and for $m = \frac{n-i+1}{2}$, we have that  $$\pi(i)\pi(i+2)\cdots\pi(n-1) = (m)(m-1)\cdots(1).$$ Then we have $$s_{\{123,132\}}^m(\pi)(i-2) = m+1$$ for some $m\in \N$. \\

Let $\pi$ be a permutation as in Claim 2. Let $a$ be such that $\pi(a) = m+1$. If $1<a<i$, then we will show that $s_{\{123,132\}}(\pi)(i-2) = m+1$. Let's look at what happens at the point when $\pi(a)$ is about to enter the stack. At this point, $\pi(1)$ is at the bottom of the stack along with at least one other entry, and $\pi(2),\dots,\pi(a-1)$ are either sorted or in the stack. Since $\pi(1),\dots,\pi(a-1)$ are all larger than $\pi(a)$, when $\pi(a)$ is about to enter the stack, first all but the bottom element of the stack must leave the stack. Thus, $\pi(a)$ enters the stack, second to bottom. Since $\pi(a+1),\dots,\pi(i)$ are all larger than $\pi(a)$, we have that $\pi(a)$ stays in the stack until $\pi(i)$ is about to get sorted, at which point, $\pi(a)$ leaves the stack. When this happens, $\pi(2),\dots,\pi(a-1),\pi(a+1),\dots,\pi(i-1)$ are all sorted. Thus, $\pi(a)$ gets moved to index $i-2$, as desired.

In the case that $\pi(a) = m+1$ and $i<a<n$, by Claim 1, in $\frac{n-i}{2}$ iterations of $s_{\{123,132\}}$, we have that $A$ is at index $i-1$ at which point the first case applies. 

In the case when $\pi(1) = m+1$, since $s_{\{123,132\}}(\pi)(n)= m+1$ after one sort, we would be in the second case.

Thus, in all cases, we have that the claim holds. \\

From Claim 1 and Claim 2, we can see that after having $s_{123,132}$ applied to it enough times, every permutation $\pi$ gets sorted to a half-decreasing permutation. \end{proof}

\begin{theorem}\label{thm: periodic}
The periodic points of $s_{\{123,132\}}$ are exactly the half-decreasing permutations.
\end{theorem}

\begin{proof}
Let $\pi$ be a half-decreasing permutation. By \cref{lem: half-decreasing}, we have that the decreasing half is fixed and all other entries are shifted to the left cyclically under $s_{\{123,132\}}$. Since there are $ \lceil\frac{n+1}{2}  \rceil $ entries not in the decreasing half, $s_{\{123,132\}}^{\left \lceil\frac{n+1}{2} \right \rceil}(\pi) = \pi$.

Let $\pi$ be not half-decreasing. Then for some $m\in\N$, $s_{\{123,132\}}^m(\pi)$ is half-decreasing. So, for all $\ell>m$, $s_{\{123,132\}}^\ell(\pi)$ is also half-decreasing. Thus, the trajectory of $\pi$ can never return to $\pi$ and so $\pi$ is not a periodic point.   
\end{proof}

\cref{thm: periodic} results in the following corollaries.

\begin{cor}\label{cor: orbit-size}
Every periodic orbit of the map $s_{123,132}:S_n\to S_n$ has size  $\left \lceil \frac{n+1}{2} \right \rceil. $
\end{cor}

\begin{cor}
The map $s_{\{123,132\}}: S_n \to S_n$ has $\lfloor\frac{n}{2}\rfloor!$ periodic orbits. 
\end{cor}

\begin{proof}
The largest element of the decreasing half of a permutation is $m = \lfloor\frac{n-1}{2}\rfloor$. Then there are $n-m$ elements not in the decreasing half of the permutation. So, since there are $(n-m)!$ many ways to arrange the elements not in the decreasing half, there are $(n-m)! = \left(\lceil \frac{n+1}{2}\rceil\right)!$ many half-decreasing permutations. There are $\lceil \frac{n+1}{2}\rceil$ many elements in each orbit and so we have that there are $$\frac{\left(\lceil \frac{n+1}{2}\rceil\right)!}{\lceil \frac{n+1}{2})\rceil} = \left(\left\lceil \frac{n+1}{2}\right\rceil-1\right)! = \left\lfloor \frac{n}{2}\right\rfloor ! $$ many orbits.   
\end{proof}

\begin{definition}
Let $\pi$ be a permutation of length $n$. We say that $\pi$ is \textit{half-increasing} if its complement is half-decreasing. 
\end{definition}

\begin{cor}
The periodic points of $s_{\{312,321\}}$ are exactly the half-increasing permutations.
\end{cor}

This corollary follows from \cref{lem: complement}.

\section{Open Problems}

We conclude with the following conjectures and questions. 

\begin{conj}
The only periodic points of $s_{\{132,213\}}$ and  $s_{\{231,213\}} $ are the identity and its reverse. 
\end{conj}

Another interesting problem to consider is the following:
\begin{ques}
Given a set of permutations $T$, can one find a classification based on $T$ of the maximum number of preimages under the map $s_T$?
\end{ques}

Something else to look into is the size of the image of $s_T$. This question was studied in depth for the original stack sort by Bousquet-Mélou  \cite{Images} and Defant \cite{asymptotic}. 

\begin{ques}
For a given set of permutations $T$ what is the size of the image of $s_T$?
\end{ques}

\section*{Acknowledgements}
This research was conducted through the Duluth REU and was funded through NSF grant 1949884 and NSA grant H98230-20-1-0009. We would like to thank Joe Gallian for the REU, and Ilani Axelrod-Freed, Colin Defant, and Mihir Singhal for helpful discussions as well as Joe Gallian, Noah Kravitz, Yelena Mandelshtam, and Colin Defant for valuable comments.


\begin{thebibliography}{99}

\bibitem{Baril}
J. L. Baril, G. Cerbai, C. Khalil and V. Vajnovszki, Catalan and Schröder permutations sortable by two restricted stacks. arXiv:2004.01812.

\bibitem{Bona}
M. Bóna, \textit{Combinatorics of Permutations} CRC Press, Inc. 2004. 

\bibitem{Images}
M. Bousquet-Mélou, Sorted and/or sortable permutations. \emph{Discrete Math.}, {\bf 225} (2000), 25--50.

\bibitem{CerbaiMachines}
G. Cerbai, Sorting Cayley permutations with pattern-avoiding machines. arXiv:2003.02536.

\bibitem{Cerbai}
G. Cerbai, A. Claesson, and L. Ferrari, Stack sorting with restricted stacks. \emph{J. Combin. Theory, Ser. A.}, {\bf 173} (2020), 105230.

\bibitem{Cerbai132}
G. Cerbai. A. Claesson, L. Ferrari, and E. Steingrímsson, Sorting with pattern-avoiding stacks: the $132$-machine. \emph{Electron. J. Combin.}, To appear. arXiv:2006.05692.

\bibitem{claesson}
A. Claesson and H. \'Ulfarsson, Sorting and preimages of pattern classes. \emph{DMTCS proc.}, {\bf AR} (2012), 595--606.

\bibitem{Fertility}
C. Defant, Fertility monotonicity and average complexity of the stack-sorting map. arXiv:2003.05935.


\bibitem{Fertility3}
C. Defant, Fertility numbers. \emph{J. Comb.}, {\bf 11} (2020), 527--548.

\bibitem{Fertility2}
C. Defant, Fertility, strong fertility, and postorder Wilf equivalence. \emph{Australas. J. Combin.}, {\bf 76} (2020), 149-182.

\bibitem{ColinPreims}
C. Defant, Preimages under the stack-sorting algorithm. \emph{Graphs and Combin.}, {\bf 33} (2016), 103–122.


\bibitem{asymptotic}
C. Defant, Troupes, cumulants, and stack-sorting. arXiv:2004.11367.


\bibitem{ColinLassalle}
C. Defant, M. Engen, and J. Miller, Stack-sorting, set partitions, and Lassalle's sequence. \emph{J. Combin. Theory, Ser. A.}, {\bf 175} (2020), 105275.

\bibitem{Knuth}
D. Knuth, \textit{The Art of Computer Programming, Volume 1 (3rd Ed.): Fundamental Algorithms}. Addison Wesley Longman Publishing Co., Inc., 1997. 


\bibitem{West}
J. West, Permutations with restricted subsequences and stack-sortable permutations, Ph.D. Thesis, MIT, 1990.




\end{thebibliography}
\end{document}